\documentclass{article}
\usepackage{amsmath,amssymb,amsthm}
\usepackage{pstricks,pst-node}
\usepackage{graphicx}
\usepackage{multirow,units}

\newtheorem{thm}{Theorem}[section]

\newtheorem{lem}[thm]{Lemma}
\newtheorem{obs}[thm]{Observation}

\newtheorem*{fact}{Fact}

\theoremstyle{definition}
\newtheorem{exa}{Example}[section]

\numberwithin{equation}{section}

\newcommand{\R}{\mathbb{R}}

\title{On fractional realizations of graph degree sequences}
\author{Michael D. Barrus\\Department of Mathematics, Brigham Young University\\Provo, UT 84602}

\begin{document}
\maketitle
\begin{abstract}
We introduce fractional realizations of a graph degree sequence and a closely associated convex polytope. Simple graph realizations correspond to a subset of the vertices of this polytope. We describe properties of the polytope vertices and characterize degree sequences for which each polytope vertex corresponds to a simple graph realization. These include the degree sequences of pseudo-split graphs, and we characterize their realizations both in terms of forbidden subgraphs and graph structure.
\end{abstract}

\section{Introduction}
A list of nonnegative integers is called graphic if it is the degree sequence of a simple graph. In the following, let $d=(d_1,\dots,d_n)$ be a graphic list, and consider a vertex set $\{1,\dots,n\}$, which we denote by $[n]$. A \emph{realization of $d$} is a simple graph with vertex set $[n]$ where each vertex $i$ has degree $d_i$. A given degree sequence may have several realizations. Many interesting questions concern these realizations, such as determining properties that these graphs may singly hold or must all hold, and finding techniques for generating all realizations or randomly selecting one.
papers on generating realizations.

Many algorithms for generating realizations first find one by using an algorithm of Havel~\cite{Havel55} and Hakimi~\cite{Hakimi62} and then use 2-switches (described later herein) or similar graph operations to obtain all other realizations. Other approaches may avoid edge-switching; see the paper by Kim et al.~\cite{KimEtAl09} for references to many algorithms and for an example of a ``degree-based'' procedure that generates realizations by systematically searching through adjacency scenarios.

This paper will approach realizations of a degree sequence from a degree-based perspective, albeit with a somewhat relaxed notion of a realization. Given $d$ and the vertex set $[n]$, we associate a variable $x_{ij}$ with each unordered pair $i,j$ of distinct vertices. Interpreting $x_{ij}=1$ to mean that vertices $i$ and $j$ are adjacent and $x_{ij} = 0$ to mean that they are not, each realization of $d$ naturally corresponds to a solution to %
\begin{align*}
\sum_{i} x_{ij} = d_j, \qquad & 1 \leq j \leq n;\\
x_{ij} \in \{0,1\}, \qquad & 1 \leq i < j \leq n,
\end{align*} %
where the sum is over all $i$ in $[n]$ other than $j$. We thus model degree sequence realizations as solutions to an integer problem.

The conditions above are typical of those found in integer programming problems. One common technique in optimization is to relax the requirement that the variables be integers; instead, we allow the variables to take on values in prescribed intervals and solve a ``fractional'' version of the problem. Fractional graph theory often models combinatorial parameters as integer problems and relaxes them in this way. Fractional analogues of these combinatorial notions have opened up a rich landscape in which classical results may be placed in broader context or given simpler proofs. A good introduction to fractional graph theory may be found in~\cite{ScheinermanUllman97}.

We now relax the integer conditions on the variables $x_{ij}$ above. Consider the set $P(d)$ of all points $x=(x_{ij})$ in $\R^{\binom{n}{2}}$ whose coordinates are lexicographically indexed by pairs $i,j$ (with $i<j$) of vertices in $[n]$ and that satisfy the conditions
\begin{align}
\label{eq: sum=deg} \sum_{i} x_{ij} = d_j, \qquad & 1 \leq j \leq n;\\
\label{eq: in [0,1]} 0 \leq x_{ij} \leq 1, \qquad & 1 \leq i < j \leq n.
\end{align}

Given a point $x$ in $P(d)$, we define the \emph{fractional realization of $d$ corresponding to $x$} to be the labeling of the edges of the complete graph on $[n]$ such that the edge $ij$ receives the label $x_{ij}$ for all pairs $i,j$ of distinct elements in $[n]$. Figure~\ref{fig: frac realizations} illustrates three fractional realizations of $(1,1,1,1,1,1)$ (for clarity, edges labeled with 0 are not shown.) As in (a), simple graph realizations of $d$ correspond naturally to fractional realizations in which each edge of the complete graph is labeled with $0$ or $1$. We refer at times to the point $x$ as the \emph{characteristic vector} of the fractional realization. We call the conditions in~\eqref{eq: sum=deg} and~\eqref{eq: in [0,1]} the \emph{degree conditions} and \emph{hypercube conditions}, respectively.

\begin{figure}
\centering
\includegraphics{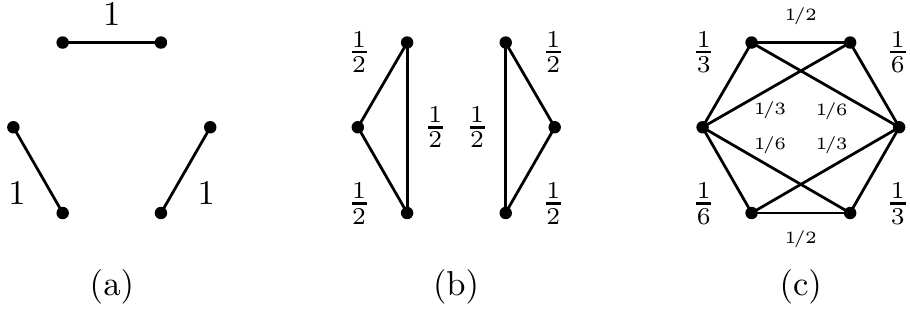}
%
\caption{Fractional realizations of $(1,1,1,1,1,1)$.}
\label{fig: frac realizations}
\end{figure}

The set $P(d)$ is a convex polytope, the convex hull of a set of points in $\R^{\binom{n}{2}}$. As such, perhaps the first question that arises about $P(d)$ is what its extreme points are. Given the origin of our formulation of $P(d)$, we also wish to see if points corresponding to simple graph realizations have any special role. We can easily answer the latter question; each $(0,1)$-vector in $P(d)$ is a vertex of the polytope, since it satisfies $\binom{n}{2}$ of the hypercube conditions with equality. In fact, for all graphic $d$ with five or fewer terms, we can easily verify with a computer algebra system that the vertices of $P(d)$ correspond precisely to simple graph realizations of $d$.

For more general $d$, however, $P(d)$ may have non-integral vertices. For example, if $d=(1,1,1,1,1,1)$, then the characteristic vector of the fractional realization in Figure~\ref{fig: frac realizations}(b) is also a vertex of $P(d)$; in fact, of the 25 vertices of $P(d)$, ten have non-integral coordinates.

Thus the vertices of $P(d)$ may or may not correspond to simple graph realizations of $d$. In this paper we study the vertices of $P(d)$ and conditions under which they have only integer coordinates.  We first characterize the extreme points of $P(d)$ in Section 2; we show that, as illustrated in Figure~\ref{fig: frac realizations}(b), these vertices are precisely those whose coordinates come from $\{0, 1/2, 1\}$, with the $1/2$-values assigned to edges that form vertex disjoint odd cycles.

We then study sequences $d$ for which $P(d)$ has no non-integral vertices. We call these sequences \emph{decisive} (since they force each $x_{ij}$ to take either 0 or 1 as its value), and we call their realizations \emph{decisive graphs}. We characterize the decisive sequences and decisive graphs in Sections 3 through 5. In Section 3, we identify decisive sequences via a forbidden configuration condition. As consequences we find that the decisive graphs form a hereditary class containing the pseudo-split graphs, and we determine a complete list of their minimal forbidden induced subgraphs. In Section 4 we obtain a structural characterization of decisive graphs that generalizes the vertex partition properties of split and pseudo-split graphs. In Section 5 this structural characterization yields another characterization of decisive sequences. We conclude with some remarks on $P(d)$ and our characterizations of decisive sequences and graphs in Section 6.

Before proceeding, we define some terms and notation. The vertex set of a graph $G$ will be denoted by $V(G)$. Given vertices $u, v \in V(G)$, we say that $u$ is a neighbor of $v$ if $u$ is adjacent to $v$. Otherwise, we may refer to $u$ as a non-neighbor of $v$ or say that $uv$ is a non-edge of $G$. Given $W \subseteq V(G)$, we use $G[W]$ to denote the induced subgraph of $G$ with vertex set $W$. The complement of a graph $G$ will be denoted by $\overline{G}$. Complete graphs, cycles, and paths with $n$ vertices will be denoted by $K_n$, $C_n$, and $P_n$, respectively. The complete bipartite graph with partite sets of sizes $a$ and $b$ is denoted by $K_{a,b}$. The house graph is defined as the complement of $P_5$.

\section{Vertices of $P(d)$}
In this section we characterize the extreme points of $P(d)$ in terms of their coordinates. As the following theorem shows, the structure exhibited by the fractional realization of $(1,1,1,1,1,1)$ in Figure~\ref{fig: frac realizations}(b) is typical of those corresponding to nonintegral vertices of $P(d)$.

\begin{thm}\label{thm: vtcs of S}
Given a graphic list $d$, let $h$ be a point of $P(d)$, and let $H$ be the fractional realization of $d$ corresponding to $h$. The point $h$ is a vertex of $P(d)$ if and only if the edges of $H$ labeled with nonintegral coordinates of $h$ form vertex disjoint odd cycles. Furthermore, there are an even number of these cycles, and the nonintegral coordinates of $h$ all equal $1/2$.
\end{thm}
\begin{proof}
Suppose that $h$ is a vertex of $P(d)$. Then $h$ is the unique point in the intersection of $\binom{n}{2}$ of the bounding hyperplanes. We may express the equations of these hyperplanes as a matrix-vector equation $Ah = b$, where $A$ is an $\binom{n}{2}$-by-$\binom{n}{2}$ matrix and $b$ is a vector in $\R^{\binom{n}{2}}$. Let $Q$ be the set of all edges of $H$ labeled with nonintegral values, and let $P$ be the set of vertices of $H$ belonging to an edge in $Q$; further let $p=|P|$ and $q=|Q|$. We show that $p=q$.

Since $h$ is the unique solution of $Ax = b$, we see that $A$ is invertible and hence has nonzero determinant. This implies that there exists a collection $T$ of $\binom{n}{2}$ nonzero entries in $A$ with no two in the same row or column. Consider an arbitrary pair $(i,j)$ such that the entry $h_{ij}$ is nonintegral. Consider the element of $T$ in the column of $A$ corresponding to the edge $ij$. The row containing this element clearly does not come from a hyperplane of the form $x_{ij}=\alpha$, where $\alpha \in \{0,1\}$; it must instead belong to a row arising from the degree condition at a vertex incident with edge $ij$. We associate this vertex (which belongs to $P$) with the edge $ij$ (which belongs to $Q$), and we similarly associate a vertex of $P$ with every other edge in $Q$. Since $T$ contains exactly one entry in each row of $A$, and $A$ has full rank, distinct edges in $Q$ must be associated with distinct vertices in $P$. Hence $p \geq q$.

Since $h$ is a vertex of $P(d)$, it satisfies all the degree conditions imposed by $d$. Thus all edges meeting at a vertex $v$ must have values that sum to an integer, and hence if some edge incident with $v$ is nonintegral, there must be another edge incident with $v$ that is also nonintegral. Hence every vertex in $P$ is incident with at least two edges in $Q$. An elementary counting argument shows that $p \leq q$, with equality if and only if each vertex in $P$ is incident with exactly two edges of $Q$. We have seen that $p \geq q$, so in fact $p=q$, and the edges of $H$ labeled with nonintegral entries of $h$ comprise a 2-regular graph $(P,Q)$. This graph, which we call $R$, is a vertex disjoint union of cycles, as claimed.

We now claim that all cycles of $R$ are odd. If $R$ contains an even cycle with edges $e_1,\dots,e_m$ in order, then let $\alpha$ denote the smaller of $\min\{h_{e_i}: 1 \leq i \leq m, \text{ }i\text{ odd}\}$ and $\min\{1-h_{e_i}: 1 \leq i \leq m, \text{ }i\text{ even}\}$. Define $h'$ to be the vector agreeing with $h$ on all coordinates except for those corresponding to $e_1,\dots,e_m$, where instead we define $h'_{e_i} = h_{e_i}-\alpha$ for odd $i$ and $h'_{e_i}=h_{e_i}+\alpha$ for even $i$. Note that $h'$ satisfies all degree conditions for $d$ and also maintains all integral entries of $h$. Thus we have $Ah' = b$, which contradicts the claim that $h$ was the unique solution to $Ax=b$. Thus $R$ contains no even cycle.

Conversely, let $g$ be a point of $P(d)$ with corresponding fractional realization $G$ such that the edges of $G$ labeled with nonintegral coordinates of $g$ form pairwise disjoint odd cycles. Suppose that there are $k$ such cycles, and that altogether they contain $q$ vertices and $q$ edges. Consider the system $Y$ of equations consisting of the equation $x_{ij}=g_{ij}$ for every edge $ij$ of $G$ labeled with an integral coordinate of $g$. Take all degree equations corresponding to vertices of $G$ incident with at least one of the edges $ij$ of $G$ for which $g_{ij}$ is nonintegral; reduce these by substituting in the values of $x_{ij}$ explicitly specified by $Y$. The resulting equations each contain two variables; augment $Y$ to include these equations. Modeling the equations of $Y$ by the matrix-vector equation $Ax=b$, we may permute the columns of $A$ to create a block diagonal matrix with the form \[\begin{bmatrix}
A_{11} & 0 & \cdots & 0\\
0 & \ddots & \ddots & \vdots\\
\vdots & \ddots & A_{kk} & 0\\
0 & \cdots & 0 & I
\end{bmatrix},\]
where $I$ indicates the identity matrix of order $\binom{n}{2}-p$, and each $A_{ii}$ has the form \[
\begin{bmatrix}
1 & 1 & 0 & \cdots & 0\\
0 & 1 & 1 & \ddots & \vdots\\
\vdots & \ddots & \ddots & \ddots & 0\\
0 & \cdots & 0 & 1 & 1\\
1 & 0 & \cdots & 0 & 1
\end{bmatrix}
\]
(i.e., the vertex-edge incidence matrix of a cycle) with order equal to the number of vertices in the $i$th cycle of nonintegrally labeled edges in $G$. It is straightforward to verify that since the order of each $A_{ii}$ is odd, $\det A_{ii} \neq 0$ for each $i$, and thus $Ax = b$ has the unique solution $x=g$. Hence $g$ is a vertex of $P(d)$.

Now with $h$ a vertex of $P(d)$ and $H$ its corresponding fractional realization, and with the graph $R=(P,Q)$ as described above, consider two consecutive edges on one of the cycles in $R$. Each has a value strictly between 0 and 1, and the two values must sum to an integer. If the first has value $\alpha$, then the second must have value $1-\alpha$. Continuing around the cycle, the edges alternately have values $\alpha$ and $1-\alpha$. However, since the cycle has odd length, we eventually see that the each edge has a value simultaneously equal to $\alpha$ and $1-\alpha$. This forces $\alpha=1/2$.

Finally, note that when the values of all edge labels in $H$ are added together, each cycle in $R$ contributes an odd multiple of $1/2$. However, the sum should equal an integer (it equals half the sum of the vertex degrees; for a graphic list the degree sum is an even integer), so there must be an even number of cycles.
\end{proof}

In light of Theorem~\ref{thm: vtcs of S}, we refer henceforth to edges of a fractional realization as $0$-edges, $1/2$-edges, or $1$-edges according to the values of the coordinates they correspond to in the associated vector.

\section{Decisive sequences and blossoms}\label{sec: three}

One consequence of Theorem~\ref{thm: vtcs of S} is our assertion in Section 1 that any degree sequence $d$ with five or fewer terms has the property that all vertices of $P(d)$ are integral. For these $d$, the vertices of $d$ correspond exactly to the simple graph realizations of $d$. For which longer degree sequences is this also the case?

Because such degree sequences require the vertices of $P(d)$ to have coordinates each equal to $0$ or $1$---the extreme values of~\eqref{eq: in [0,1]}, and not anything in between---we call them \emph{decisive sequences}. As we will see, their realizations, the \emph{decisive graphs}, satisfy strict structural properties. In the remainder of the paper we characterize the decisive sequences and graphs. We handle these characterizations in three steps. In Section~\ref{sec: three} we show that a degree sequence is decisive if and only if none of its realizations contains a certain pattern of adjacencies and non-adjacencies known as a $(3,3)$-blossom. In Section~\ref{sec: four} we focus on decisive graphs and prove the equivalence of forbidding a $(3,3)$-blossom, forbidding each of a list of 70 potential induced subgraphs, and being able to partition the vertex set of a graph into three sets meeting certain adjacency properties. The strict structure these graphs possess allows us to return in Section~\ref{sec: five} to their degree sequences, this time characterizing the decisive sequences in terms of the numerical values of their terms.

We begin with some definitions. Given $d$, let $H$ be a fractional realization of $d$. Given odd integers $k,\ell \geq 3$, we define a \emph{fractional $(k,\ell)$-blossom} in $H$ to be a configuration on $k+\ell$ vertices $v_1,\dots,v_k,w_1,\dots,w_\ell$ in which the vertex pairs in \[\{v_1v_2, v_2v_3,\dots,v_{k-1}v_k, v_kv_1\} \text{ and } \{w_1w_2,w_2w_3,\dots,w_{\ell-1}w_\ell, w_\ell w_1\}\] are all $1/2$-edges, and the vertex pair $v_1w_1$ is either a $0$-edge or a $1$-edge. We denote this configuration by $(v_2,\dots,v_k,v_1; w_1,\dots,w_{\ell})$.

We further define an \emph{integral $(k,\ell)$-blossom} to be a configuration in $H$ on $\{v_1,\dots,v_k,w_1,\dots,w_\ell\}$ in which the vertex pairs \[v_1v_2, v_2v_3,\dots,v_{k-1}v_k, v_kv_1, v_1w_1, w_1w_2,w_2w_3,\dots,w_{\ell-1}w_\ell, w_\ell w_1\] are alternately $0$-edges and $1$-edges, with $v_1w_1$ either a $0$-edge or a $1$-edge. We denote this configuration by $[v_2,\dots,v_k,v_1; w_1,\dots,w_{\ell}]$. Possibilities for both fractional and integral $(3,3)$-blossoms are illustrated in Figures~\ref{fig: frac blossoms} and~\ref{fig: int blossoms}, where the $0$-edges, $1/2$-edges, and $1$-edges are represented by dotted, dashed, and solid lines, respectively.

\begin{figure}
\centering
\includegraphics{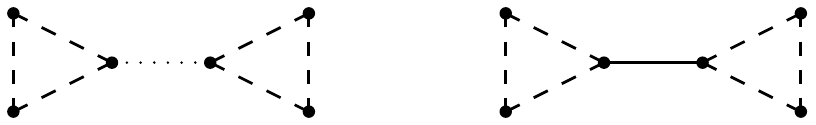}
\caption{Fractional $(3,3)$-blossoms}
\label{fig: frac blossoms}
\end{figure}
\begin{figure}
\centering
\includegraphics{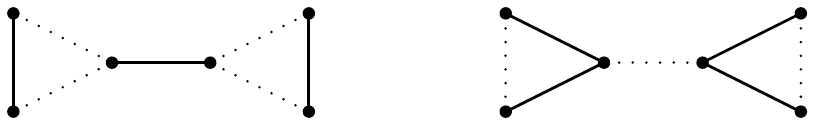}
\caption{Integral $(3,3)$-blossoms}
\label{fig: int blossoms}
\end{figure}

We now present our first characterization of decisive sequences. An \emph{integral realization} of $d$ is a fractional realization in which all edge labels are 0 or 1, as in Figure~\ref{fig: frac realizations}(a).

\begin{thm} \label{thm: decisive vs blossoms}
For a graphic sequence $d$, the following are equivalent:
\begin{enumerate}
\item[\textup{(1)}] $d$ is a decisive sequence;
\item[\textup{(2)}] No integral realization of $d$ contains an integral $(k,\ell)$-blossom for any odd $k,\ell \geq 3$;
\item[\textup{(3)}] No integral realization of $d$ contains an integral $(3,3)$-blossom.
\end{enumerate}
\end{thm}
\begin{proof}
\emph{(1) implies (2):} We prove the contrapositive. Suppose that $d$ has an integral realization $G$ containing an integral $(k,\ell)$-blossom \[[v_2,\dots,v_k,v_1;w_1,\dots,w_\ell]\] for some odd $k$ and $\ell$ such that $k,\ell \geq 3$. Let $H$ be a fractional realization of $d$ obtained by replacing the $k+\ell+1$ edges of this $(k,\ell)$-blossom with the edges of the fractional $(k,\ell)$-blossom $(v_2,\dots,v_k,v_1;w_1,\dots,w_\ell)$, with the edge $v_1w_1$ assigned the value $1-\alpha$, where $\alpha$ is the value of $v_1w_1$ in $G$. All other edges of $H$ receive the same value as in $G$. We claim that the characteristic vector $h$ of $H$ is a vertex of $d$, showing that $d$ is not decisive. By Theorem~\ref{thm: vtcs of S}, it suffices to show that $h$ is in $P(d)$. This is straightforward to verify, as all coordinates of $h$ satisfy the hypercube conditions, and the replacement of integral edges at each vertex by $1/2$-edges does not change the sum of edge values at that vertex, meaning that $H$ is a fractional realization of $d$, as claimed.

\emph{(2) implies (3):} Immediate.

\emph{(3) implies (1):} If $d$ is not a decisive sequence, then by Theorem~\ref{thm: vtcs of S} it contains a vertex $h$ corresponding to a fractional realization $H$ with at least two disjoint odd cycles of $1/2$-edges. Let $v_1,\dots,v_k$ and $w_1,\dots,w_\ell$ be the vertices of these respective cycles. Let $H'$ be a fractional realization of $d$ obtained by replacing the $k+\ell+1$ edges of the fractional $(k,\ell)$-blossom \[(v_2,\dots,v_k,v_1; w_1,\dots,w_\ell)\] with those of the integral $(k,\ell)$-blossom $[v_2,\dots,v_k,v_1; w_1,\dots,w_\ell]$, such that the edge $v_1w_1$ receives the value $1-\alpha$, where $\alpha$ is the value of $v_1w_1$ in $H$. All other edges of $H'$ receive the same value as in $H$. It is straightforward to verify that $H'$ is also a fractional realization of $d$, though it contains strictly fewer nonintegral edges. If we iteratively carry out switches similar to the one just described, we will eventually arrive at a realization $G$ of $d$ having no edges with nonintegral labels. The last switch performed creates an integral $(k,\ell)$-blossom in $G$; for convenience assume that it is $[v_2,\dots,v_k,v_1; w_1,\dots,w_\ell]$. Suppose that $v_1w_1$ is a 1-edge in $G$, so each of $v_1v_2$, $v_1v_k$, $w_1w_2$, and $w_1w_\ell$ is a non-edge. Either $v_2v_k$ is an edge or $v_2v_3,\dots,v_{k-1}v_k,v_kv_2$ is a sequence that alternates between 1-edges and 0-edges; switching the 0s and 1s assigned to these edges produces another integral realization of $d$ in which $v_2v_k$ is an edge. Similarly, either $w_2w_\ell$ is an edge or we may switch the 1s and 0s assigned to edges along the cycle $w_2,\dots,w_\ell,w_2$ to create a realization in which $w_2w_\ell$ is an edge. It follows that in some integral realization of $d$ there is an integral $(3,3)$-blossom $[v_2,v_k,v_1;w_1,w_2,w_k]$. A similar argument holds, with 1-edges and 0-edges exchanging roles, if $v_1w_1$ is a 0-edge of $G$.
\end{proof}

\section{Decisive graphs} \label{sec: four}

Theorem~\ref{thm: decisive vs blossoms} establishes the equivalence of a degree sequence being decisive and containing an integral $(3,3)$-blossom in none of its integral realizations. An immediate consequence of this is that we may change our framework of study in two ways. First, since integral realizations of $d$ correspond exactly with simple graph realizations of $d$, we may recognize whether the extreme points of $P(d)$ all have integral coordinates by examining simple graphs instead of fractional realizations. We use the term \emph{$(3,3)$-blossom} to mean a configuration analogous to an integral $(3,3)$-blossom, where edges and non-adjacencies replace 1-edges and 0-edges, respectively.

Second, instead of dealing with labeled graphs, where different realizations of a degree sequence were treated as distinct, we now may treat graphs in the same isomorphism class as the same. This is because whether a graph contains a $(3,3)$-blossom is determined by its isomorphism class and not by which element of that class it is.

Another consequence of Theorem~\ref{thm: decisive vs blossoms} is that the decisive graphs form a hereditary class, i.e., one closed under taking induced subgraphs, as we show below. This property may not be apparent from the definition, since decisive graphs are defined in terms of degree sequences $d$, which in turn are defined based on their polytopes $P(d)$.

Given a graph $G$, an \emph{alternating 4-cycle} in $G$ is a configuration on four vertices $\{a,b,c,d\}$ where $ab$ and $cd$ are edges and $bc$ and $da$ are not edges. A \emph{2-switch} is an operation on a graph that takes such an alternating 4-cycle, deletes edges $ab$ and $cd$, and adds edges $bc$ and $da$ to the graph. In studying simple graph realizations of degree sequences, the following lemma of Fulkerson, Hoffman, and McAndrew~\cite{FulkersonEtAl65} is a fundamental tool.

\begin{lem}[\cite{FulkersonEtAl65}] \label{lem: FHM}
Two graphs on the same vertex set have the same degree sequence if and only if one can be obtained from the other via a finite sequence of 2-switches.
\end{lem}

\begin{lem} \label{lem: hereditary}
The class of decisive graphs is hereditary, i.e., closed under taking induced subgraphs.
\end{lem}
\begin{proof}
Suppose that $H$ is a graph that is not decisive. By Theorem~\ref{thm: decisive vs blossoms} and Lemma~\ref{lem: FHM}, there exists a sequence of 2-switches that produces a graph on the same vertex set that contains a $(3,3)$-blossom. If $G$ is any graph containing $H$ as an induced subgraph, then this same sequence of 2-switches, applied to the induced subgraph $H$, creates a $(3,3)$-blossom in $G$, making $G$ not decisive.
\end{proof}

In light of Lemma~\ref{lem: hereditary}, the class of decisive graphs, like all hereditary classes, has a characterization in terms of forbidden induced subgraphs. We use Theorem~\ref{thm: decisive vs blossoms} to begin the search for the forbidden subgraphs: any graph $G$ containing a $(3,3)$-blossom is forbidden, and the proof of Lemma~\ref{lem: hereditary} shows that any graph having the same degree sequence as $G$ is also forbidden.

Beginning with the $(3,3)$-blossoms in Figure~\ref{fig: int blossoms} (assuming now that solid lines represent edges and dotted lines represent non-adjacencies), we consider all possible ways of adding edges. The degree sequences of the resulting graphs are the following:

\begin{equation}\label{eq: forb deg seq}
\begin{array}{cccc}
(1, 1, 1, 1, 1, 1), & (3, 3, 2, 2, 1, 1), & (4, 2, 2, 2, 1, 1), & (4, 4, 3, 3, 2, 2),\\
(2, 2, 1, 1, 1, 1), & (3, 3, 2, 2, 2, 2), & (4, 3, 2, 2, 2, 1), & (4, 4, 3, 3, 3, 1),\\
(2, 2, 2, 2, 1, 1), & (3, 3, 3, 2, 2, 1), & (4, 3, 3, 2, 2, 2), & (4, 4, 3, 3, 3, 3),\\
(2, 2, 2, 2, 2, 2), & (3, 3, 3, 3, 1, 1), & (4, 3, 3, 3, 2, 1), & (4, 4, 4, 3, 3, 2),\\
(3, 2, 2, 1, 1, 1), & (3, 3, 3, 3, 2, 2), & (4, 3, 3, 3, 3, 2), & (4, 4, 4, 4, 3, 3),\\
(3, 2, 2, 2, 2, 1), & (3, 3, 3, 3, 3, 3), & (4, 4, 2, 2, 2, 2), & (4, 4, 4, 4, 4, 4).
\end{array}
\end{equation}

For the rest of the paper, let $\mathcal{B}$ denote the set of all graphs having a degree sequence listed in~\eqref{eq: forb deg seq} (regardless of whether these graphs can be obtained by adding edges to a $(3,3)$-blossom). We remark that we deliberately avoid drawing or otherwise listing the 70 graphs that comprise $\mathcal{B}$. While presenting their degree sequences is certainly more convenient, we also note that our proof of Theorem~\ref{thm: decisive graph equivalences} will refer more commonly to the degree sequences of certain induced subgraphs, rather than to the specific isomorphism class of a realization. We shall also have more to say later about the notion of ``forbidden degree sequences'' in Section~\ref{sec: six}.

Given a graph class $\mathcal{F}$, we say a graph $G$ is \emph{$\mathcal{F}$-free} if no induced subgraph of $G$ is isomorphic to an element of $\mathcal{F}$. In the following, the symbols $+$ and $\vee$ respectively indicate a disjoint union and a join. The graph $U$ is the unique graph with degree sequence $(4,2,2,2,2,2)$; it and its complement $\overline{U}$ are illustrated in Figure~\ref{fig: U}. A graph $G$ is \emph{split} if its vertex set can be partitioned into an independent set $V_1$ and a clique $V_2$.
\begin{figure}
\centering
\includegraphics{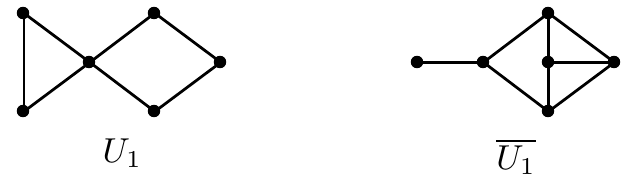}
\caption{The graphs $U$ and $\overline{U}$}
\label{fig: U}
\end{figure}

\begin{thm} \label{thm: decisive graph equivalences}
Let $d$ be a graphic list. The following are equivalent and characterize decisive sequences and graphs:
\begin{enumerate}
\item[\textup{(1)}] None of the realizations of $d$ contains a $(3,3)$-blossom.
\item[\textup{(2)}] Every realization of $d$ is $\mathcal{B}$-free.
\item[\textup{(3)}] $d$ has a $\mathcal{B}$-free realization.
\item[\textup{(4)}] $d$ has a realization $G$ for which there exists a partition $V_1,V_2,V_3$ of $V(G)$ such that
\begin{enumerate}
\item[\textup{(i)}] $V_1$ is an independent set and $V_2$ is a clique;
\item[\textup{(ii)}] each vertex in $V_3$ is adjacent to every vertex of $V_2$ and to none of the vertices in $V_1$; and
\item[\textup{(iii)}] $G[V_3]$ is split or has fewer than six vertices or is one of $U$, $\overline{U}$, $K_2+K_{1,m}$, or $(K_{m}+K_1)\vee 2K_1$ for some $m \geq 3$.
\end{enumerate}
\item[\textup{(5)}] Every realization of $d$ has the form described in (4).
\end{enumerate}
\end{thm}

\noindent In the rest of this section we prove (1) $\Rightarrow$ (2) $\Rightarrow$ (3) $\Rightarrow$ (4) $\Rightarrow$ (5) $\Rightarrow$ (1).

\bigskip
\noindent \emph{\textup{(1)} implies \textup{(2)}:} Let $G$ be a realization of $d$. 
Suppose $G$ contains an element $F$ of $\mathcal{B}$ as an induced subgraph. By definition, there is some graph $F'$  in $\mathcal{B}$ that has the same degree sequence as $F$ and contains a $(3,3)$-blossom. By Lemma~\ref{lem: FHM}, there exists a sequence of 2-switches that produces $F'$ from $F$; performing these 2-switches in $G$ produces a $(3,3)$-blossom in $G$, a contradiction.

\bigskip
\noindent \emph{\textup{(2)} implies \textup{(3)}:} Immediate, since $d$ is graphic.

\bigskip
\noindent \emph{\textup{(3)} implies \textup{(4)}:} We first define a useful notion. Call $G$ \emph{decomposable} if there exist sets $V_1,V_2,V_3$ partitioning $V(G)$ satisfying (i) and (ii) of (4), with the additional requirement that $V_1 \cup V_2$ and $V_3$ are both nonempty. Graphs that are not decomposable are \emph{indecomposable}.

Decomposable graphs have appeared in the work of several authors with varying terminology and notation. Notably, in~\cite{Tyshkevich00} R. Tyshkevich developed the idea of decomposability to produce a canonical decomposition of graphs that has a useful analogue in terms of degree sequences; we will encounter some of these ideas in the next section. For our purposes, the following weaker decomposition will be sufficient.

\begin{thm}[\cite{Tyshkevich00}]\label{thm: unique decomp}
For every graph $G$ there is a partition $V_1,V_2,V_3$ of $V(G)$ satisfying \textup{(i)} and \textup{(ii)} of Theorem~\ref{thm: decisive graph equivalences}(4) such that $V_3$ is not empty and $G[V_3]$ is indecomposable. This partition is unique in the sense that if $V_1,V_2,V_3$ and $V'_1,V'_2,V'_3$ are two partitions with these properties, then $G[V_3]$ and $G[V_3']$ are isomorphic, and there is an isomorphism from $G[V_1\cup V_2]$ to $G[V'_1 \cup V'_2]$ that bijectively maps $V_1$ onto $V_1'$ and $V_2$ onto $V_2'$.
\end{thm}

Assume now that $G$ is an arbitrary $\mathcal{B}$-free graph, and let $V_1,V_2,V_3$ be a partition of $V(G)$ as in Theorem~\ref{thm: unique decomp}. Let $H=G[V_3]$; by assumption, $H$ is indecomposable. Assume that $H$ is not split (otherwise, $G$ is split) and contains at least six vertices. We must show that $H$ is one of the graphs listed in part (iii) of condition (4).

\begin{fact} If $H$ contains an induced subgraph isomorphic to any element of \[\{C_5,P_5,\textup{house}, K_2 + K_3, K_{2,3}\},\] then $H$ is equal to that subgraph.
\end{fact}
\begin{proof}
Let us suppose that $H$ contains an induced 5-cycle  $v_1v_2v_3v_4v_5v_1$;. If $w$ is a vertex of $H$ not on the 5-cycle, then $w$ must be adjacent to all or none of the vertices in $C=\{v_1,v_2,v_3,v_4,v_5\}$, since otherwise $G[C \cup \{w\}]$ has one of \[(3,2,2,2,2,1),\; (3,3,2,2,2,2),\; (3,3,3,3,2,2), \text{ or }(4,3,3,3,3,2)\] as its degree sequence and hence belongs to $\mathcal{B}$. Of the vertices of $H$ not in $C$, let $A$ denote those having no neighbor in $C$, and let $B$ denote those adjacent to every vertex of $C$. Since both $K_2+P_4$ and $2K_1 \vee P_4$ are elements of $\mathcal{B}$ (their degree sequences appear in~\eqref{eq: forb deg seq}) and $C_5$ induces $P_4$, the sets $A$ and $B$ must be an independent set and a clique, respectively. However, then $H$ is decomposable with vertex set partition $A,B,C$, a contradiction, unless $A \cup B$ is empty. hence $H \cong C_5$.

Similar arguments apply if $H$ contains an induced subgraph isomorphic to $P_5$ or the house graph. Note also that if a graph $J$ is formed by adding a vertex and some edges to $K_2+K_3$ or $K_{2,3}$, then a graph $J'$ with the same degree sequence as $J$ can be formed by adding a vertex and some edges to $P_5$ or the house graph, respectively, and the same number of edges will be added to produce $J'$ as were added for $J$. It follows that if $H$ contains an induced subgraph isomorphic to any element of $\{C_5,P_5,\textup{house}, K_2 + K_3, K_{2,3}\}$, then $H$ is equal to that subgraph.
\end{proof}

Since each element of $\{C_5,P_5,\textup{house}, K_2 + K_3, K_{2,3}\}$ has five vertices, by the previous fact we may assume for convenience that $H$ induces none of these subgraphs. F\"{o}ldes and Hammer~\cite{FoldesHammer76} showed that a graph is split if and only if it is $\{2K_2,C_4,C_5\}$-free. Since $H$ is not split, it must contain $2K_2$ or $C_4$ as an induced subgraph. Assume that $H$ induces $2K_2$, and let $ab$ and $cd$ be the edges of an induced copy of $2K_2$. Let $C=\{a,b,c,d\}$.

\begin{fact}
Any vertex of $H$ not in $C$ is adjacent to exactly 1 or 3 vertices in $C$.
\end{fact}
\begin{proof}
If any vertex $w$ of $H$ is adjacent to exactly two vertices from $C$, then $H[C \cup \{w\}]$ is isomorphic to either $P_5$ or $K_2+K_3$, a contradiction. Let $A$ be the set of all vertices of $H$ having no neighbors in $C$, and let $B$ the set of all vertices adjacent to every vertex in $C$. If $t$ is a vertex of $A$ and $u$ is a neighbor of $t$, then $H[C \cup \{t,u\}]$ has a degree sequence from~\eqref{eq: forb deg seq} unless $u$ is adjacent to every vertex of $C$, forcing $u \in B$. Similarly, if $v$ is a vertex of $B$ and $w$ is a vertex of $H$ not adjacent to $v$, then $w$ cannot have any neighbor in $C$ and hence belongs to $A$. It follows that $A,B,V(H)-(A\cup B)$ is a partition of $V(H)$ showing $H$ to be decomposable unless $A=B=\emptyset$.
\end{proof}

For $k \in \{1,3\}$, let $N_k$ denote the set of vertices of $H-C$ that have exactly $k$ neighbors in $C$.

\begin{fact}
The vertices in $N_1$ form an independent set and all have the same neighbor in $C$.
\end{fact}
\begin{proof}
If $v$ and $w$ are vertices of $N_1$ with differing neighbors in $C$, then $H[C \cup \{v,w\}]$ has degree sequence $(2,2,1,1,1,1)$ or $(2,2,2,2,1,1)$ and hence belongs to $\mathcal{B}$, a contradiction. If $v$ and $w$ are adjacent and have the same neighbor in $C$, then $H[C \cup \{v,w\}]$ has degree sequence $(3,2,2,1,1,1)$, another contradiction.
\end{proof}

\begin{fact}
The vertices in $N_3$ form an independent set and all have the same three neighbors in $C$.
\end{fact}
\begin{proof}
If $v$ and $w$ are vertices of $N_3$ that differ on their neighbors in $C$, then $H[C \cup \{v,w\}]$ has degree sequence $(3,3,3,3,2,2)$ or $(4,4,3,3,2,2)$ and hence belongs to $\mathcal{B}$. If $v$ and $w$ are adjacent and have the same neighbors in $C$, then $H[C \cup \{v,w\}]$ has degree sequence $(4,4,3,3,3,1)$ and thus belongs to $\mathcal{B}$.
\end{proof}

\begin{fact}
$|N_3| \leq 2$.
\end{fact}
\begin{proof}
Suppose that $t,u,v$ are distinct vertices in $N_3$, and without loss of generality assume that these are all adjacent to $\{a,b,c\}$. Then $H[\{a,b,c,t,u,v\}]$ has degree sequence $(4,4,3,3,3,3)$ and hence belongs to $\mathcal{B}$, a contradiction.
\end{proof}

Without loss of generality, we may assume that all vertices in $N_1$ are adjacent to $a$.

\begin{fact}
If $N_1$ and $N_3$ are both nonempty, then $H$ is isomorphic to $U$.
\end{fact}
\begin{proof}
Let $u$ be an arbitrary vertex of $N_1$, and let $v$ be an arbitrary vertex of $N_3$. We claim first that $v$ is adjacent to $b$, $c$, and $d$; if not, then $H[C \cup \{u,v\}]$ has $(3,3,2,2,1,1)$ or $(4,3,2,2,2,1)$ as its degree sequence and hence belongs to $\mathcal{B}$. We also have that $v$ is adjacent to $u$; otherwise, $H[C \cup \{u,v\}]$ has degree sequence $(3,2,2,2,2,1)$ and thus belongs to $\mathcal{B}$.

Now if $u_1,u_2$ are vertices of $N_1$ and $v$ is a vertex of $N_3$, then $H[\{a,b,c,u_1,u_2,v\}]$ has degree sequence $(4,3,2,2,2,1)$ and belongs to $\mathcal{B}$. If $u$ is in $N_1$ and $v_1,v_2$ are vertices of $N_3$, then $H[\{a,b,c,u,v_1,v_2\}]$ has degree sequence $(3,3,3,3,2,2)$ and hence belongs to $\mathcal{B}$. Thus $|N_1|=|N_3|=1$, and $H$ is isomorphic to $U$.
\end{proof}

Recall our assumption that $H$ has at least six vertices. If $|N_3|=0$, then $H$ is isomorphic to $K_2+K_{1,m}$, where $m=|N_1|+1$ and $m \geq 3$. If $|N_3|=1$, then $N_1$ is nonempty and hence $H$ is isomorphic to $U$. Otherwise $|N_3|=2$ and $H$ is isomorphic to $\overline{U}$.

These complete the cases for our assumption that $H$ induces $2K_2$. If instead $H$ induces $C_4$, then consider $\overline{H}$; it is an indecomposable graph on at least six vertices that induces $2K_2$. It is also $\mathcal{B}$-free and $\{C_5,P_5,\textup{house}, K_2 + K_3, K_{2,3}\}$-free (note that both these sets are closed under complementation). By the arguments above, $\overline{H}$ is isomorphic to one of the graphs listed in the previous paragraph. This means that $H$ is isomorphic to one of $U$, $\overline{U}$, or $(K_m + K_1) \vee 2K_1$ for some $m \geq 3$.

\bigskip
\noindent \emph{(4) implies (5):} By Lemma~\ref{lem: FHM} every realization of $d$ can be obtained by performing a sequence of 2-switches on $G$. It suffices to prove that if $G'$ is a graph resulting from a single 2-switch on $G$, then $G'$ has the structure described in (4). By considering required adjacencies in decomposable graphs, we see that the four vertices involved in any 2-switch must all belong to either $V_1 \cup V_2$ or to $V_3$. Any 2-switch having vertices in $V_1 \cup V_2$ must have its alternate edges and nonedges each involving one vertex from $V_1$ and one vertex from $V_2$; the 2-switch therefore leaves $V_1$ an independent set and $V_2$ a clique, and it cannot change which vertices any vertex in $V_3$ may be adjacent to. The same is true for any 2-switch whose vertices all belong to $V_3$. Thus $G'$ satisfies properties (i) and (ii) of (4). Furthermore, the preceding arguments about 2-switches with vertices in $V_1\cup V_2$ also show that if $G[V_3]$ is split, then any 2-switch with vertices belonging to $V_3$ leaves $G'[V_3]$ also split. Since $G[V_3]$ is indecomposable and 2-switches preserve decomposability, $G'[V_3]$ is also indecomposable. Note also $G[V_3]$ contains the same (number of) vertices as $G'[V_3]$. Finally, observe that performing a 2-switch on any member of $\{U, \overline{U}, K_2 + K_{1,m}, (K_m + K_1) \vee 2K_1\}$ (where $m \geq 3$) preserves the isomorphism class of the member. Thus condition (iii) also holds for $G'$.

\bigskip
\noindent \emph{(5) implies (1):} Let $G$ be a realization of $d$, and suppose $G$ contains a $(3,3)$-blossom. Let $V_1,V_2,V_3$ be as described in (4). Suppose first that the $(3,3)$-blossom is of the type shown on the left in Figure~\ref{fig: int blossoms} having three edges and four non-edges. Let $u$ and $v$ denote the vertices in the figure's center. Neither $u$ nor $v$ can belong to $V_2$; indeed, each is non-adjacent to two vertices, which thus belong to $V_1$, which contradicts that $V_1$ is an independent set. Nor can one of $u$ or $v$ belong to $V_1$, since the other would then belong to $V_2$.

Thus both $u$ and $v$ belong to $V_3$. Consider the two non-neighbors $s,t$ of $u$ in the $(3,3)$-blossom. Since both are nonadjacent to $u$, neither can belong to $V_2$. Thus $s$ and $t$ belong to $V_1 \cup V_3$; since no vertex in $V_1$ has any neighbor in $V_1 \cup V_3$, we conclude that $s$ and $t$ belong to $V_3$. A similar argument applies to the two non-neighbors of $v$. Hence all vertices of the $(3,3)$-blossom belong to $V_3$. If $G$ instead contains a $(3,3)$-blossom of the type shown on the right in Figure~\ref{fig: int blossoms}, then similar arguments, with $V_1$ and $V_2$ trading roles, again show that all the $(3,3)$-blossom vertices belong to $V_3$.

This is a contradiction, since $G[V_3]$ cannot contain a $(3,3)$-blossom, as we now show. We know that $G[V_3]$ has the form specified in (4); no graph on fewer than six vertices is large enough to contain a $(3,3)$-blossom. No split graph can contain a $(3,3)$-blossom, for the same reasons given above that $G[V_1 \cup V_2]$ (which is a split graph) could not contain a $(3,3)$-blossom. Finally, it is a simple matter to verify that none of $U$, $\overline{U}$, $K_2+K_{1,m}$ or $(K_m+K_1) \vee 2K_1$ contains a $(3,3)$-blossom. Thus no realization of $d$ contains a $(3,3)$-blossom, and this completes the proof of Theorem~\ref{thm: decisive graph equivalences}.

\bigskip
We note in closing that the structure of the $\mathcal{B}$-free graphs presented in (4) of Theorem~\ref{thm: decisive graph equivalences} generalizes the structure of $\{2K_2,C_4\}$-free graphs, known also as the pseudo-split graphs~\cite{MaffrayPreissmann}. In~\cite{BlazsikEtAl93}, Bl\'{a}zsik et al.~showed that a graph is $\{2K_2,C_4\}$-free if and only if it is split or has a partition $V_1$, $V_2$, $V_3$ of its vertex set such that conditions (i) and (ii) of (4) above hold, and $V_3$ is the vertex set of an induced $C_5$. Decisive graphs thus also include the split graphs ($\{2K_2,C_4,C_5\}$-free graphs~\cite{FoldesHammer76}) and threshold graphs ($\{2K_2,C_4,P_4\}$-free graphs~\cite{ChvatalHammer73}).

\section{A degree sequence characterization} \label{sec: five}
We now use the structure of decisive graphs given in condition (4) of Theorem~\ref{thm: decisive graph equivalences} to characterize their degree sequences.

Our characterization will involve the well known Erd\H{o}s--Gallai inequalities. Given a list $\pi = (\pi_1,\dots,\pi_n)$ of nonnegative integers in nonincreasing order, the \emph{$k$th Erd\H{o}s--Gallai inequality} is the statement \[\sum_{i=1}^k \pi_i \leq k(k-1) + \sum_{i>k} \min\{k,d_i\}.\] Erd\H{o}s and Gallai~\cite{ErdosGallai60} showed that $\pi$ is the degree sequence of a simple graph if and only if $\pi$ has even sum and satisfies the Erd\H{o}s--Gallai inequalities for all $k \in \{1,\dots,n\}$. We observe that by evaluating an empty sum as $0$, the $0$th Erd\H{o}s--Gallai inequality holds with equality for all graphic lists.

\begin{thm} \label{thm: decisive seqs}
Let $d=(d_1,\dots,d_n)$ be a graphic list in weakly decreasing order. Let $k$ be the largest integer such that $d$ satisfies the $k$th Erd\H{o}s--Gallai inequality with equality. The list $d$ is a decisive sequence if and only if one of the following is true:
\begin{enumerate}
\item[\textup{(1)}] $k = \max\{i:d_i \geq i-1\}$;
\item[\textup{(2)}] the number $\ell = \max \{i:d_i \geq k \text{ and } i>k\}$ exists and satisfies one of
\begin{enumerate}
\item[\textup{(i)}] $\ell - k \leq 5$;
\item[\textup{(ii)}] $(d_{k+1}-k, \dots, d_\ell - k)$ is one of \[(3, 3, 3, 3, 3, 1), (4, 2, 2, 2, 2, 2), (m, 1^{(m+2)}), ((m + 1)^{(m+2)}, 2)\] where $m \geq 3$.
\end{enumerate}
\end{enumerate}
\end{thm}

We prove Theorem~\ref{thm: decisive seqs} in the remainder of this section. We proceed by showing that the conditions in (1) and (2) are equivalent to the cases in statement (iii) in Theorem~\ref{thm: decisive graph equivalences}.

Given an arbitrary graph $G$, let $V_1,V_2,V_3$ be vertex sets partitioning $V(G)$ as in Theorem~\ref{thm: unique decomp}. We observe that $G$ is split if and only if $G[V_3]$ is split. Hammer and Simeone~\cite{HammerSimeone81} gave a characterization of split graphs in terms of their degree sequences.

\begin{thm}[\cite{HammerSimeone81}]
Let $G$ be a graph, and let $(d_1,\dots,d_n)$ be its degree sequence in weakly decreasing order. The graph $G$ is a split graph if and only if \[\sum_{i=1}^m d_i = m(m-1) + \sum_{i>m} d_i,\] where $m = \max\{i:d_i \geq i-1\}$.
\end{thm}

With $m=\max\{i:d_i \geq i-1\}$, note that for $i>m$ we have $d_i \leq d_{m+1} < m$, so $d_i = \min\{m,d_i\}$. We also have the following.

\begin{lem}[{\cite[Corollary 5.5]{Barrus13}}]\label{lem: j leq m}
If the $j$th Erd\H{o}s--Gallai inequality holds with equality then $j \leq m$.
\end{lem}

Hence a graph is split if and only if $k = m$. Thus condition (1) in Theorem~\ref{thm: decisive seqs} is equivalent to the first part of condition (4)(iii) of Theorem~\ref{thm: decisive graph equivalences}.

We move now to the condition (2) in Theorem~\ref{thm: decisive seqs}. In~\cite{Barrus13} the author described the relationship between the canonical decomposition of a degree sequence (see~\cite{Tyshkevich00}) and equalities among the Erd\H{o}s--Gallai inequalities. As mentioned in the previous section, the canonical decomposition is a finer vertex partition than the partition $V_1$, $V_2$, $V_3$ defined above. With $k$ and $\ell$ defined as above, applying the results of~\cite{Barrus13} to the current context yields the following:

\begin{thm}[{\cite[Theorem 5.6]{Barrus13}}]\label{thm: k and ell}
Let $G$ be a graph with degree sequence $d=(d_1,\dots,d_n)$ and vertex set $\{v_1,\dots,v_n\}$, indexed so that $d_G(v_i)=d_i$. Suppose that $G$ is  decomposable with vertex partition $V_1,V_2,V_3$ as defined above, with $G[V_3]$ indecomposable. If $G$ is not split, then the clique $V_2$ is equal to the set $\{v_i : i \leq k\}$. In this case $V_1$ is precisely the set $\{v \in V(G): d_G(v)<k\}$.
\end{thm}

Now assume that a realization $G$ with degree sequence $d$ and the usual decomposition $V_1,V_2,V_3$ is not split. By Lemma~\ref{lem: j leq m}, $d_m \geq m-1 \geq k$. Thus $\ell \geq m$, and by Theorem~\ref{thm: k and ell}, $\ell - m = |V_3|$. Thus $G[V_3]$ has fewer than six vertices if and only if $\ell - k \leq 5$. Furthermore, $(d_{k+1}-k,\dots,d_\ell - k)$ is the degree sequence of $G[V_3]$. Thus $G[V_3]$ is isomorphic to one of $U$, $\overline{U}$, $K_2+K_{1,m}$, or $(K_{m}+K_1)\vee 2K_1$ for some $m \geq 3$ if and only if $(d_{k+1}-k,\dots,d_\ell - k)$ is one of $(3, 3, 3, 3, 3, 1)$, $(4, 2, 2, 2, 2, 2)$, $(m, 1^{(m+2)})$, or $((m + 1)^{(m+2)}, 2)$ for some $m \geq 3$ (note that these graphs are the unique realizations, up to isomorphism, of their respective degree sequences). We have now shown the equivalence of the conditions (1) and (2) in Theorem~\ref{thm: decisive seqs} to condition (4)(iii) in Theorem~\ref{thm: decisive graph equivalences}.

Since condition (4) in Theorem~\ref{thm: decisive graph equivalences} characterizes realizations of decisive sequences, the proof of Theorem~\ref{thm: decisive seqs} is complete.

\section{Remarks} \label{sec: six}

We have here considered a polytope $P(d)$ arising naturally in the study of fractional realizations of a degree sequence $d$. We have characterized both the vertices of $P(d)$ and the degree sequences $d$ for which the vertices of the polytope correspond precisely to integral realizations of $d$. Since $P(d)$ is a bounded convex polytope, each vertex achieves the optimal value of some linear objective function in a linear program. As a possible direction for future study, we ask whether these objective functions may be used to conveniently identify individual realizations or illustrate their properties.

We remark that our characterization of decisive sequences and graphs in Theorem~\ref{thm: decisive vs blossoms} has an interesting form, in that $d$ is a decisive sequence if none of its (possibly many) realizations contains a certain configuration. This is reminiscent of a partial order $\preceq$ defined by Rao~\cite{Rao80} on the set of all graph degree sequences, in which $e \preceq d$ if there exists some realization of $d$ containing some realization of $e$. Restating part of Theorem~\ref{thm: decisive graph equivalences}, we have the following.

\begin{obs}~\label{obs: Rao}
A degree sequence $d$ is decisive if and only if for every sequence $e$ in~\eqref{eq: forb deg seq} we have $e \npreceq d$.
\end{obs}

Theorem~\ref{thm: decisive seqs} is then an explicit characterization of the degree sequences that satisfy this partial order requirement. Chudnovsky and Seymour recently proved~\cite{ChudnovksySeymour} that $\preceq$ defines a well quasiorder, that is, given any infinite list of degree sequences, there exist two sequences in the list that are comparable under the relation $\preceq$. This implies that in any characterization in terms of ``forbidden degree sequences'' such as the one in Observation~\ref{obs: Rao}, a minimal list of forbidden sequences must be finite; as an illustration, our list in~\eqref{eq: forb deg seq} has 24 degree sequences.

Moving beyond decisive sequences, a number of questions remain about the polytope $P(d)$ for general $d$; we conclude with two. First, in defining $P(d)$ we considered only hyperplanes arising from the degree conditions and the hypercube conditions. As we noted in Section~1, these hyperplanes may create vertices of $P(d)$ that are non-integral. For such $d$, which additional hyperplanes are needed to ``trim off'' fractional vertices, resulting in a polytope that is the convex hull of the realization vertices? What combinatorial meaning do the corresponding inequalities have?

Second, because $P(d)$ is a subset of $\R^{\binom{n}{2}}$ containing points constrained by $n$ degree conditions, we note that $P(d)$ in general might have dimension as large as $(n^2-3n)/2$. However, for some $d$ the dimension is in fact much smaller. For instance, threshold sequences are the graphic sequences having a unique labeled realization, such as $(4,3,2,2,1)$, and these realizations are called threshold graphs. (For a monograph on these sequences and graphs and an extensive bibliography, see~\cite{MahadevPeled95}.) Suppose that $d$ is threshold. Note that every point of a convex polytope may be written as a convex combination of vertices of the polytope. Since threshold sequences are decisive sequences, as we observed in Section~\ref{sec: four}, we have the following observation.

\begin{obs}\label{obs: threshold}
If $d$ is a threshold sequence, then $d$ has a unique fractional realization, i.e., $P(d)$ consists of a single point.
\end{obs}

Thus for threshold sequences the dimension of $P(d)$ equals 0, and we ask for other properties of $d$ that restrict the dimension of $P(d)$. Observe, for example, that the dimension of $P(d)$ decreases whenever the degree and hypercube conditions uniquely determine the value of a variable $x_{ij}$. In light of Observation~\ref{obs: threshold}, this is what happens when $d$ is threshold, but it may happen for more general sequences; when $d=(2,2,1,1)$, the polytope conditions force $x_{12}=1$ and $x_{34}=0$, and $P(d)$ has dimension $1$, rather than $2$. In~\cite{Barrus14} the author showed that for simple graphs the corresponding forced adjacencies and non-adjacencies among vertices are preserved as one proceeds higher in the majorization partial order on fixed-length graphic partitions of an even integer. The forced adjacency relationships culminate with the threshold sequences, the maximal graphic partitions~\cite{PeledSrinivasan89}, where all edges and non-edges are uniquely determined. We therefore ask: given general degree sequences $d$ and $e$ such that $d$ majorizes $e$, is it true that the dimension of $P(d)$ is less than or equal to the dimension of $P(e)$?

\end{document}